\documentclass{article}
\usepackage{amsmath,amsthm,amsfonts,graphicx}
\usepackage{multirow}
\usepackage{slashbox}
\usepackage{multirow}
\def\smallddots{\mathinner{\raise7pt\hbox{.}\raise4pt\hbox{.}\raise1pt\hbox{.}}}
\def\smallsdots{\mathinner{\raise1pt\hbox{.}\raise4pt\hbox{.}\raise7pt\hbox{.}}}

\DeclareMathOperator{\diag}{diag}

\DeclareMathOperator{\rank}{rank}

\numberwithin{equation}{section}
\numberwithin{table}{section}
\newtheorem{theorem}{Theorem}[section]

\newtheorem{corollary}{Corollary}[section]
\newtheorem{fact}{Fact}[section]

\newtheorem{example}{Example}[section]
\newtheorem{definition}{Definition}[section]

\newtheorem{remark}{Remark}[section]

\begin{document}
\title{\bf How Bad Are Vandermonde Matrices?}
\author{Victor Y. Pan  \\
%}
%$^{[1, 2],[a]}$ 
%\and\\
%$^{[1]}$ 
Department of Mathematics and Computer Science \\
Lehman College of the City University of New York \\
Bronx, NY 10468 USA \\
%$^{[2]}
and  Ph.D. Programs in Mathematics  and Computer Science \\
The Graduate Center of the City University of New York \\
New York, NY 10036 USA \\
%$^{[a]}$ 
victor.pan@lehman.cuny.edu \\
http://comet.lehman.cuny.edu/vpan/  \\
} 
\date{}

\maketitle

\begin{abstract}  
The work on the estimation of the condition numbers of Van\-der\-monde matrices, motivated by applications to interpolation and quadrature, can be traced back at least to the 1970s. Empirical study has shown consistently that Van\-der\-monde matrices tend to be badly ill-conditioned, with a narrow class of notable exceptions, such as the matrices of the discrete Fourier transform (hereafter referred to as DFT). So far formal support for this empirical observation, however, has been limited to the  matrices defined by the real set of  knots. We prove that, more generally, any Van\-der\-monde matrix of a large size is badly ill-conditioned unless its knots are more or less equally spaced on or about the circle  $C(0,1)=\{x:~|x|=1\}$. The matrices of DFT are  perfectly conditioned, being defined by a cyclic sequence of knots, equally spaced on that circle, but we prove that even a slight modification of the knots  into the so called quasi-cyclic sequence on this circle defines badly ill-conditioned Van\-der\-monde matrices. Likewise we prove that the
half-size leading block of a large DFT matrix is badly ill-conditioned. (This result was motivated by an application to pre-conditioning of an input matrix for Gaussian elimination with no pivoting.) 
Our analysis involves the Ekkart--Young theorem,  the Van\-der\-monde-to-Cauchy transformation of matrix structure, our new inversion formula for
 a Cauchy matrix, and  low-rank approximation of its large submatrices. 
\end{abstract}



%------------------------------------------------------------------------------

\section{Introduction}\label{s1}

%------------------------------------------------------------------------------

Our main subject, motivated by applications to interpolation  and quadrature,
 is the estimation of the condition number, $\kappa(V_{\bf s})$,
of a nonsingular $n\times n$ Van\-der\-monde matrix, $V_{\bf s}$,
where
\begin{equation}\label{eqvand0}
\kappa(V_{\bf s})=||V_{\bf s}||~||V_{\bf s}^{-1}||,~~~~~~~~~~~
V_{\bf s}=(s_i^{j})_{i,j=0}^{n-1},
\end{equation} 
 ${\bf s}=(s_i)_{i=0}^{n-1}$
denotes the  vector of $n$ distinct knots $s_0,\dots,s_{n-1}$,
and $||M||=||M||_2$ denotes  the spectral norm of a matrix $M$.
First of all, one would like  to know 
whether this number is reasonably bounded or large, that is,
whether the matrix is well- or ill-conditioned.
 
It has been proven in  \cite{GI88}, \cite{G90}, and
 \cite{T94} that the condition number $\kappa(V_{\bf s})$
is exponential in $n$
 if all knots $s_0,s_1,\dots,s_{n-1}$ are real.
This means that such a  matrix 
is badly ill-conditioned
already for moderately large integers $n$,
 but empirically almost all
 Van\-der\-monde matrices of  large or even reasonably large sizes 
 are badly
 ill-conditioned as well, with a narrow class of exceptions,
such as the 
 Van\-der\-monde matrices
defined by  cyclic sequences of knots, equally spaced on the circle
$C(0,1)=\{x:~|x|=1\}$. Such matrices 
are unitary up to scaling by $1/\sqrt n$
and thus are perfectly conditioned, that is, have the 
minimum condition number 1.
A celebrated example is the
matrix
 of the discrete Fourier
transform (DFT),
\begin{equation}\label{eqdft}
\Omega=(\omega_n^{ij})_{i,j=0}^{n-1},~{\rm for}~\omega_q=\exp(2\pi \sqrt{-1}/q)
\end{equation} 
denoting a primitive $q$th root of 1.  

A slight   modification of the knots into 
the so called
quasi-cyclic sequence can make the matrices 
badly ill-conditioned, however.
This has been an 
empirical  observation of \cite[Section IV]{G90},
whose formal support 
 turned out to be elusive.
The critical step was the estimation of the norm of the inverse matrix,
$||V_{\bf s}^{-1}||$.
The known expressions for the entries of this matrix are not explicit enough.
In particular, write
\begin{equation}\label{eqvndinv}
s(x)=\prod_{j=0}^{n-1}(x-s_j),
~s_i(x)=s(x)/((x-s_i)s_j'(s_i))=\sum_{i,j=0}^{n-1}s_{i,j}x^j,
\end{equation} 
and recall \cite[equation (2.3)]{G90} that
\begin{equation}\label{eqnvndinv}
V_{\bf s}^{-1}=(s_{i,j})_{i,j=0}^{n-1},~{\rm and~so}~
||V_{\bf s}^{-1}||\ge \max_{i,j=0}^{n-1}|s_{i,j}|
\end{equation}
(see other expressions 
in \cite{M03}, \cite{SEMC11}, and 
the references therein).   
As Walter Gautschi wrote in
\cite[Section IV]{G90},  
 the known lower and upper bounds on 
the condition number of a
 Van\-der\-monde matrix, "unfortunately, are too
far apart to give much useful information``, 
%even in  the quasi cyclic case,
and so he resorted to 
``computing the condition numbers $\dots$
numerically``. 

Our new lower bounds
enable us to support formally the cited empirical observations
for a very general class of Van\-der\-monde matrices.
Our results (see Sections \ref{scv} and \ref{svndbd}) 
indicate that
{\em the condition number of an $n\times n$
 Van\-der\-monde matrix is 
 exponential in $n$ unless its knots are 
more or less 
equally spaced on or about the unit circle} $C(0,1)$.

Our study covers
 the matrices defined  with the quasi-cyclic sequence of knots
 as a special case
(see Section \ref{skntunit}), and we  
further observe in Section \ref{sgenpv}
 that all $k\times k$ blocks of the $n\times n$
matrix of discrete Fourier transform,
for $n=2k$, 
are badly ill-conditioned as long as the integers $k$ and
$n$ are sufficiently large,
even though the $n\times n$ DFT matrix itself is unitary up to scaling. 

%our methods yield formal proof 
%(see Section \ref{skntunit}).
%Similar techniques enable us to prove that the
%half-size leading bloc of the very well 
%conditioned matrix of DFT is badly ill-conditioned
%(see Section \ref{sgenpv}).
%This result was motivated by an application to pre-conditioning with random circulant %multipliers.

In order 
to prove our lower bounds on the condition number $\kappa(V_{\bf s})$, 
we shift from
 Van\-der\-monde 
matrices
$V_{\bf s}$ of (\ref{eqvand0})
to Cauchy matrices 
\begin{equation}\label{eqcauchy}
C_{\bf s,t}=\Big (\frac{1}{s_i-t_j}\Big )_{i,j=0}^{n-1},
\end{equation}
defined by $2n$ distinct knots  $s_0,\dots,s_{n-1}$; $t_0,\dots,t_{n-1}$.
The matrices of these two classes
are much different, and it is not obvious that their studies  
can be linked to one another. 
Unlike a Van\-der\-monde 
matrix
$V_{\bf s}$, a Cauchy matrix
$C_{\bf s,t}$ has rational entries, 
its every submatrix is a Cauchy matrix as well, and
the  Cauchy structure is invariant in column interchange
as well as in scaling and shifting all entries by the same scalar,
that is, 
$aC_{a{\bf s},a{\bf t}}=C_{\bf s,t}$ for $a\neq 0$ and 
$C_{{\bf s}+a{\bf e},{\bf t}+a{\bf e}}=
C_{\bf s,t}$ for ${\bf e}=(1,\dots,1)^T$.
Moreover, 
%and this is more relevant to our current study,
some powerful technique has been developed in
the papers \cite{MRT05}, \cite{R06},
 \cite{P15}, \cite{Pa}, \cite{XXG12}, \cite{XXCB14}
for the approximation of
reasonably large submatrices of Cauchy matrices
by lower-rank matrices.

Nevertheless there is 
a simple  Van\-der\-monde--Cauchy link, 
which has become the basis for our progress in 
 this paper as well as in  \cite{P15} and \cite{Pa}.
(The 
 paper \cite{P89/90} has revealed
a more general link among the matrix structures of Van\-der\-monde, Cauchy,
Toeplitz and Hankel  types, based on their displacement
representation. The paper has also pointed out 
 potential algorithmic applications, and 
this work has indeed become the springboard for devising  the
highly efficient algorithms of the subsequent papers
 \cite{GKO95}, \cite{G98}, \cite{MRT05}, \cite{R06},
 \cite{XXG12}, \cite{XXCB14},  \cite{P15}, and \cite{Pa}).)

Restatement of our original task in terms of Cauchy matrices
has  open new opportunities for our study of
Van\-der\-monde matrices,  
and we obtained the desired progress by extending the known formula 
for the Cauchy determinant to the following simple expression for 
all entries of the inverse of a Cauchy
matrix $C_{\bf s,t}$,

\begin{equation}\label{eqij}
(C_{\bf s,t}^{-1})_{i',j'}= 
\frac{(-1)^{n}}{t_{j'}-s_{i'}}s(t_{j'})t(s_{i'})~
{\rm for}~s(x)=\prod_{i=0}^{n-1}(x-s_i)~{\rm of}~(\ref{eqvndinv}),~
~t(x)=\prod_{j=0}^{n-1}(x-t_j),
{\rm and}~{\rm all~pairs}~(i',j'). 
\end{equation}
 We applied this expression in order to prove exponential 
lower bounds on the norms of the inverses 
and the condition numbers of Cauchy matrices
of a large class,
and we extended these bounds to Van\-der\-monde matrices.
We have also deduced alternative exponential 
lower bounds by applying the known techniques 
for low-rank approximation of certain large submatrices of  
   Cauchy matrices.

In Section \ref{scv0}, by applying  the Ekkart--Young theorem,
we more readily prove the
exponential growth of the condition number 
of the Van\-der\-monde 
matrices  of a more narrow class, namely, the ones 
 having at least a single knot 
 outside the circle $C(0,1)$ 
or having many knots inside this circle.
These results provide simple initial insight into the subject.

Otherwise we organize our presentation as follows.
In the next section we  recall some basic definitions and auxiliary results.
In Section \ref{scv} at first 
we deduce our basic lower bounds on 
the condition number $\kappa(V_{\bf s})$
by applying a link 
between
Van\-der\-monde and  Cauchy matrices;
then we informally interpret
these bounds in terms of some restrictions on the
knot set $\{s_0,s_1,\dots,s_{n-1}\}$
of any well-conditioned Van\-der\-monde
matrix. 
In Sections \ref{skntunit} and \ref{sgenpv} we prove 
 exponential growth of the condition number
of 
 Van\-der\-monde matrices, defined by the 
quasi-cyclic sequence of knots and the half-size 
leading blocks of the DFT matrices,
respectively.  In Section \ref{scv1} we
deduce an exponential lower bound on the condition number 
of the Cauchy matrices closely linked to 
Van\-der\-monde matrices. 
In Section \ref{svndbd} we extend this bound to 
an alternative exponential lower bound on 
the condition number $\kappa(V_{\bf s})$.
This bound 
is weaker than the bounds of Section \ref{scv},
for the matrix classes of Sections  \ref{skntunit} and \ref{sgenpv}, 
but its domain of applications is shown more explicitly. 
In Section \ref{ststs} we present the results of our numerical tests. 
Section \ref{sconc} is left for conclusions.

%------------------------------------------------------------------------------

\section{Some Definitions and Auxiliary Results}\label{sdef}

%------------------------------------------------------------------------------

$\sigma_j(M)$ denotes the $j$th largest singular value of an
$n\times n$ matrix $M$.
%$M=(m_{ij})_{i,j=0}^{m-1,n-1}$ of rank $\rho=\rank(M)$, for
%$j=1,\dots,\rho$.

 $||M||=||M||_2=\sigma_1(M)$ is its spectral norm.

$\kappa(M)=\sigma_1(M)/\sigma_{\rho}(M)$ is its condition number
provided that  $\rho=\rank(M)$. 

$\sigma_{\rho}(M)=||M^{-1}||$ and $\kappa(M)=||M||~||M^{-1}||$
if $\rho=n$, that is, 
if $M$ is a nonsingular matrix.

%$\rank_{\nu}(M)$, for  $0\le \nu\le ||M||$,
%is the maximal rank of its approximation within a norm bound $\nu$,
%and so $\rank(M)=\rank_{0}(M) \le \rank_{\nu}(M)\le \rank_{\nu'}(M)$ 
%if $0\le\nu\le \nu'$.  

\begin{theorem}\label{thsgm} {\rm (The Ekkart--Young Theorem.)}
$\sigma_j(M)$ is equal to the distance $||M-M_{j-1}||$
%(measured in the  spectral norm)
between $M$ and its closest approximation 
by a matrix $M_{j-1}$ of rank at most $j-1$.
\end{theorem}

\begin{corollary}\label{coappnd}   
%\cite[Theorem 8.6.3]{GL13}.)
Suppose that a matrix $B$ has been obtained by appending $k$
new rows or $k$ new columns to a matrix $A$.
Then $\sigma_j(A)\ge \sigma_{j+k}(B)$ for all $j$.
\end{corollary}

\begin{fact}\label{fadft}
For the DFT matrix $\Omega$ 
and $\omega_n$ of (\ref{eqdft}), it holds that 

(i) $||\Omega||=||\Omega^H||=\sqrt n$ and

(ii) $n\Omega^{-1}=\Omega^H=
(\omega_n^{-ij})_{i,j=0}^{n-1}$. 
\end{fact}

%------------------------------------------------------------------------------

\section{The Condition of Van\-der\-monde Matrices: Simple Bounds}\label{scv0}

%------------------------------------------------------------------------------

%Hereafter assume Van\-der\-monde matrices $V_{\bf s}=(s_i^{j})_{i,j=0}^{m-1,n-1}$
%with $m$ distinct knots $s_0,\dots,s_{m-1}$, recall that they have full rank,

For a nonsingular Van\-der\-monde matrix $V_{\bf s}$ of (\ref{eqvand0}), write
\begin{equation}\label{eqcond}
 s_+=\max_{i=0}^{n-1} |s_i|,~|V_{\bf s}|=\max\{1,s_+^{n-1}\}.
\end{equation}

Next we prove that
the condition number $\kappa(V_{\bf s})$ is exponential in
$n$ if $s_+\ge \nu>1$ (see Theorem \ref{thnrm}, part (i)) 
or in $k$
if $1/|s_i|\ge \nu>1$
for at least $k$ knots $s_i$, $i=i_1,\dots,i_k$ 
(see Corollary \ref{cocndvceasy}).

\begin{theorem}\label{thnrm} 
For a Van\-der\-monde matrix $V_{\bf s}=(s_i^{j})_{i,j=0}^{n-1}$,  it holds that

(i) $|V_{\bf s}|\le||V_{\bf s}||\le n ~ |V_{\bf s}|$
~~~~and~~~~~
(ii) $||V_{\bf s}||\ge \sqrt n.$
\end{theorem}
\begin{proof}
The theorem readily follows from \cite[equations (2.3.8) and (2.3.11)]{GL13}.
\end{proof}

\begin{theorem}\label{thsmsv}
For a Van\-der\-monde matrix $V_{\bf s}=(s_i^{j})_{i,j=0}^{n-1}$,  it holds that

(i)  $\sigma_n(V_{\bf s})\le \sqrt n$ and

\medskip

(ii)  $\sigma_n(V_{\bf s})\le \nu^{1-k}\sqrt k\max \{k,\nu/(\nu-1)\}$ 
if $1/|s_i|\ge \nu>1$ for 
 $i=0,1,\dots,k-1$, $k\ge 1$.
 
\end{theorem}
\begin{proof}
At first   turn the matrix $V_{\bf s}$ into a rank deficient matrix
by setting
to 0 the entries of its first column.
 Then note that the norm of this perturbation equals $\sqrt n$ and
 deduce part (i)  by applying Theorem \ref{thsgm}.
Likewise, turn the matrix $V_{\bf s}$ into a rank deficient matrix by setting
to 0 its entries that lie in its first $k$ rows,
but not in its first $k-1$ columns. 
 
This  perturbation,
$\diag(s_i^{k-1})_{i=0}^{k-1}(s_i^j)_{i,j=0}^{k-1,k}$,
has column norm, $||\cdot||_1$, 
at most 
$\nu^{1-k}k$ and
has row norm, $||\cdot||_{\infty}$,  
 at most 
$\nu^{1-k}\sum_{i=k}^n\nu^{k-i}\le 
\frac{\nu^{2-k}}{\nu-1}$. So its spectral norm 
is at most $\nu^{1-k}\sqrt k \max \{k,\nu/(\nu-1)\}$.
Hence Theorem \ref{thsgm} implies part (ii). 
\end{proof}

\begin{corollary}\label{cocndvceasy}
For a Van\-der\-monde matrix $V_{\bf s}=(s_i^{j})_{i,j=0}^{n-1}$,

(i)  $\kappa(V_{\bf s})\ge \max\{1,s_+^{n-1}/\sqrt n\}$
(cf. Table \ref{VandCond}) and 
 
(ii) $\kappa(V_{\bf s})\ge |V_{\bf s}|\nu^{k-1}/(\sqrt k\max \{k,\nu/(\nu-1)\})$
if  $1/|s_i|\ge \nu>1$ for 
 $i=0,1,\dots,k-1$ (cf. Table \ref{VandCond1}).
\end{corollary}
\begin{proof}
Combine Theorems \ref{thnrm} and \ref{thsmsv} with equation (\ref{eqcond}).
\end{proof}

\begin{remark}\label{renrm}  
Observe that $||V_{\bf s}||_{\infty}= (s_+^n-1)/(s_+-1)$ for $s_+\neq 1$,
and so $||V_{\bf s}||\ge (s_+^n-1)/((s_+-1)\sqrt n)$ 
(cf.  \cite[equation (2.3.7)]{GL13}).
This refines slightly equation (\ref{eqcond}) 
and hence Corollary \ref{cocndvceasy}.
\end{remark}

 Corollary \ref{cocndvceasy}
shows that an $n\times n$ Van\-der\-monde matrix $V_{\bf s}$
is badly ill-conditioned already for a moderately large integer $n$ 
if
$s_+\ge \nu$ or  if
$1/|s_i|\ge \nu$, for a constant $\nu>1$,
 $i=0,1,\dots,k-1$, and a reasonably large integer $k$.
Therefore a Van\-der\-monde matrix $V_{\bf s}$
of a large size is badly ill-conditioned  unless
 all knots $s_0,\dots,s_{n-1}$
lie in or near the disc $\{x:~|x\le 1|\}$
and unless they lie mostly on or near its boundary cirle $C(0,1)$.
In the following  we 
prove 
stronger restrictions  
on  the  knots
of a well-conditioned  Van\-der\-monde matrix.

%but we cover general case, using no results of this section.

%------------------------------------------------------------------------------

%\begin{theorem}\label{thsngv}

%\end{theorem}

%------------------------------------------------------------------------------

\section{The Condition Estimates via the Inversion Formulae}\label{scv}

%------------------------------------------------------------------------------

Let $M_{i,j}$ denote the $(i,j)$th minor
of a nonsingular $n\times n$ matrix $M$,
that is, the determinant of its submatrix, 
obtained  by deleting its $i$th row and $j$th column
of the matrix $M$. Then
  \begin{equation}\label{eqadj}
M^{-1}=(-1)^{i+j}(M_{j,i}/\det (M))_{i,j}.
\end{equation}

Next we estimate the norm of the 
inverse of a nonsingular Cauchy matrix
$C_{\bf s,t}$ of (\ref{eqcauchy}),
by combining equation (\ref{eqadj})
with the following well known formula,

\begin{equation}\label{eqdet}
%\det (V_{\bf s})=\prod_{i<j}(s_j-s_i)~~~~~
%{\rm and}~~~~~
\det (C_{\bf s,t})=\prod_{i<j}(s_j-s_i)(t_i-t_j)/\prod_{i,j}(s_i-t_j).
\end{equation}

Since every square submatrix of a nonsingular Cauchy matrix is  
again a nonsingular Cauchy matrix,  we can extend
equations (\ref{eqdet}) to its minors, combine
these extensions with equation (\ref{eqdet}) itself and with
(\ref{eqadj}), and deduce
(after some simplifications) that 

\begin{equation}\label{eqdet1}
(C_{\bf s,t}^{-1})_{i',j'}=\frac{(-1)^{n^2+i'+j'}}{(s_{i'}-t_{j'})}
\prod_{i}(s_i-t_{i'})\prod_j(s_{j'}-t_j)/\Big (~
\prod_{i<j'}(s_{j'}-s_i)(t_i-t_{j'})
\prod_{i'<j}(s_j-s_{i'})(t_{i'}-t_j)\Big)
\end{equation}
for every pair $(i',j')$.
In particular, note that $n^2+n-1$ is an odd number and obtain
\begin{equation}\label{eqinvc}
(C_{\bf s,t}^{-1})_{n-1,0}= \frac{1}{(t_{0}-s_{n-1})}
\prod_{i}(s_i-t_{0})\prod_j(s_{n-1}-t_j)=
\frac{(-1)^{n}}{(t_{0}-s_{n-1})}s(t_0)t(s_{n-1}),
\end{equation}
for the polynomials $s(x)=\prod_{i=0}^{n-1}(x-s_i)$  
and $t(x)=\prod_{i=0}^{n-1}(x-t_i)$ of (\ref{eqij}).
By re-enumerating the knots $s_i$ and $t_j$, we can turn any pair of them 
into  the pair of $s_{n-1}$ and $t_{0}$
and thus arrive at expression (\ref{eqij}).

We are going to combine this  expression with
the following equation
(cf.  \cite[equation (5)]{P15}), which links 
Van\-der\-monde and Cauchy matrices,

\begin{equation}\label{eqvs}
V_{\bf s}=DC_{\bf s,t}D_1^{-1}V_{\bf t}~{\rm and}~
V_{\bf s}^{-1}=V_{\bf t}^{-1}D_1C_{\bf s,t}^{-1}D^{-1}.
\end{equation}
Here 
\begin{equation}\label{eqDt}
D=\diag(t(s_i))_{i=0}^{n-1},~D_1=\diag(t'(t_j))_{j=0}^{n-1},
\end{equation}
for the polynomial $t(x)=\prod_{i=0}^{n-1}(x-t_i)$ of (\ref{eqij}).

Equations (\ref{eqinvc})--(\ref{eqDt}) combined
can be considered inversion formulae 
for a  nonsingular Van\-der\-monde matrix  $V_{\bf s}$
where $t_0,\dots,t_{n-1}$ are $n$ parameters of our choice.
Let us express these parameters and consequently 
the inverse matrix $V_{\bf s}^{-1}$ through a single
 complex  nonzero parameter $f$. Write
$t_j=f\omega_n^j$, for  $j=0,\dots,n-1$, 
and thus
define the subclass of the {\em CV matrices},
$C_{{\bf s},f}=(\frac{1}{s_i-f\omega_n^j})_{i,j=0}^{n-1}$,
in the class of Cauchy matrices $C_{\bf s,t}$.
Substitute $C_{\bf s,t}=C_{{\bf s},f}$, $V_{\bf t}=V_f=\Omega\diag(f^j)_{j=0}^{n-1}$,  and $t(x)=x^{n}-f^n$
into equation (\ref{eqvs}), and obtain
%\begin{equation}\label{eqfs}
$$V_{\bf s}=D_fC_{{\bf s},f}D_{1,f}^{-1}\Omega\diag(f^j)_{j=0}^{n-1},$$
%\end{equation}
for 
%\begin{equation}\label{eqDtf}
$$D_f=
\diag(s_i^n-f^n)_{i=0}^{n-1}~{\rm and}~D_{1,f}=n\diag((f\omega_n^j)^{n-1})_{j=0}^{n-1}.$$
%\end{equation}

Combine these equations with Fact \ref{fadft},
substitute $\omega_n^{(n-1)j}=\omega_n^{-j}$ for all $j$,
and obtain
\begin{equation}\label{eqvinv}
V_{\bf s}^{-1}=\diag(f^{n-1-j})_{j=0}^{n-1}\Omega^{H}
\diag(\omega_n^{-j})_{j=0}^{n-1} 
%\diag(\omega^{-j})_{j=0}^{n-1}
C_{{\bf s},f}^{-1}\diag\Big (\frac{1}{s_i^n-f^n}\Big )_{j=0}^{n-1}.
\end{equation}
where the matrix $C_{{\bf s},f}^{-1}$ is defined by equation (\ref{eqdet1})
for $t_j=f\omega_n^j$ and $j=0,\dots,n-1$.

Equations (\ref{eqdet1}), (\ref{eqvinv}), and $t_j=f\omega_n^j$, for $|f|=1$
and $j=0,\dots,n-1$, combined can be considered inversion formulae 
for a nonsingular Van\-der\-monde matrix $V_{\bf s}$
depending on the single nonzero parameter $f$.

Furthermore we can substitute $t(x)=x^n-f^n$ into
equation  (\ref{eqinvc}) and express any entry 
$(C_{\bf s,t}^{-1})_{i',j'}$
of the matrix
$C_{{\bf s},f}^{-1}$ as follows,

\begin{equation}\label{eqinvcf}
(C_{{\bf s},f}^{-1})_{i',j'}=(-1)^n s(t_{j'})(s_{i'}^n-f^n)/(s_{i'}-t_{j'}).
\end{equation}

Equations 
(\ref{eqvinv}) and (\ref{eqinvcf}), combined together 
for all pairs $i'$ and $j'$,
 provide a simplified representation of the inverses 
of CV and Van\-der\-monde matrices. Next we apply this representation
in order to estimate the condition  number of a  Van\-der\-monde matrix. 

\begin{theorem}\label{thninv}
Let $V_{\bf s}$ be a nonsingular $n\times n$ Van\-der\-monde 
matrix of (\ref{eqvand0}).
Fix a scalar $f$ such that $|f|=1$ and 
$s_i\neq f\omega_n^j$ for all pairs of $i$ and $j$,
thus implying that
the CV matrix 
$C_{{\bf s},f}$ is
well defined and nonsingular.
Then  
  $\kappa(V_{\bf s})\ge \sqrt n~||C_{{\bf s},f}^{-1}||/\max_{i=0}^{n-1}|s_i^n-f^n|$.
\end{theorem}
\begin{proof}
%The knots $s_0,\dots,s_{n-1}$ are distinct because 
%the Van\-der\-monde matrix $V_{\bf s}$ is nonsingular
%by assumption. Clearly the knots $t_j=f\omega_n^j$, for $j=0,\dots,n-1$, 
%are distinct as well, and so the CV matrix 
%$C_{{\bf s},f}$ is  nonsingular.
Apply equation (\ref{eqvinv}) and note that the matrix 
$\diag(f^{n-1-j})_{j=0}^{n-1}\Omega^{H}
\diag(\omega_n^{-j})_{j=0}^{n-1}$  
is unitary. Therefore
$||V_{\bf s}^{-1}||=
||C_{{\bf s},f}^{-1}\diag\Big (\frac{1}{s_i^n-f^n}\Big )_{j=0}^{n-1}||$.
Hence
$||V_{\bf s}^{-1}||\ge 
||C_{{\bf s},f}^{-1}||/\max_{i=0}^{n-1}|s_i^n-f^n|$.
Combine this estimate with the bound 
$\kappa(V_{\bf s})\ge ||V_{\bf s}^{-1}||\sqrt n$ of Theorem \ref{thnrm}.
\end{proof}

\begin{corollary}\label{coninv0}
Let $V_{\bf s}$ be a nonsingular $n\times n$ Van\-der\-monde matrix.
Then 

 $$\kappa(V_{\bf s})\ge 
\sqrt n~\max_{|f|=1}|s(f)|/\max_{0\le i\le n-1}|s_i-f|~{\rm for}~s(x)~ 
{\rm of~ (\ref{eqij}).}$$ 
\end{corollary}
\begin{proof}
An infinitesimal perturbation of the parameter $f$ 
can ensure that the matrix $C_{{\bf s},f}$ is well defined,
and so we can apply Theorem \ref{thninv}.
Substitute (\ref{eqinvcf}) for $t_{j'}=f$ and for $i'$ that maximizes 
the value $|s_{i'}^n-f^n|$. 
\end{proof}

\begin{corollary}\label{coninv}
Assume  a nonsingular $n\times n$ Van\-der\-monde matrix  $V_{\bf s}$
with $s_+\le 1$ for $s_+$ of Corollary \ref{cocndvceasy}.
Write $v_i=s(\omega_{n}^i)$, $i=0,\dots,n$, and ${\bf v}=(v_i)_{s=0}^n$.
Then

$\kappa(V_{\bf s})\ge\sqrt n~\max_{|f|=1}|s(f)|/2\ge ||{\bf v}||/2.$
\end{corollary}
\begin{proof}
Note that $|s_i-f|\le 2$ for $|s_i|\le s_+\le 1$
and that $\max_{|f|=1}|s(f)|\ge ||{\bf v}||_{\infty}\ge ||{\bf v}||/\sqrt n$.
\end{proof}

Corollary \ref{coninv} suggests that 
a Van\-der\-monde matrix
$V_{\bf s}$ of a large size is badly ill-conditioned 
if the knot set $\{s_0,\dots,s_{n-1}\}$ is sparse  
in some disc that covers a reasonably thick neighborhood of 
a  reasonably long arc of the circle $C(0,1)$. 
Thus, informally, the knots of a 
well-conditioned Van\-der\-monde matrix must be
more or less equally spaced on or about  
the circle $C(0,1)$.

Towards a more formal support of the latter informal claim, one 
should estimate the growth of the coefficients 
of the polynomial $s(x)$ when the knots $s_i$
deviate from such a distribution. 
We estimate this growth in some 
special cases in the next two sections.
Then, based on  
part (i) of Theorem \ref{thninv}
and low-rank approximation of CV matrices,
 we prove
rather general results of this kind,
although only for the inputs of large sizes.

%------------------------------------------------------------------------------

We conclude this section by estimating the condition number 
$\kappa(V_{\bf s})$ in terms of the coefficients of
the polynomial $s(x)$ rather than its values.

\begin{corollary}\label{coninv2}
Under the assumptions of Corollary \ref{coninv},
write
$s(x)=x^n+\sum_{i=0}^{n} f_ix^i$, $f_n=1$, and ${\bf f}=(f_i)_{i=0}^n$. Then
$\kappa(V_{\bf s})\ge 0.5 ||{\bf f}||\sqrt {n+1}$.
\end{corollary}
\begin{proof}
Note that 
${\bf v}=\Omega_{n+1} {\bf f}$ for 
the vector ${\bf v}$ of Corollary \ref{coninv} and 
$\Omega_{n+1}=(\omega_{n+1})_{i=0}^{n}$.
Therefore 
${\bf f}=\Omega_{n+1}^{-1} {\bf v}=\frac{1}{n+1}\Omega_{n+1}^H{\bf v}$.
Consequently
$||{\bf v}||=||{\bf f}||\sqrt {n+1}$ because
$\frac{1}{\sqrt {n+1}}\Omega_{n+1}^H$ is a unitary matrix.
Substitute this equation into   Corollary \ref{coninv}.
\end{proof}

%------------------------------------------------------------------------------

For a fixed set $\{s_0,\dots,s_{n-1}\}$,
one can compute the coefficients $\sigma_0,\dots,\sigma_{n-1}$ of $s(x)$
by using $O(n\log^2(n))$ arithmetic operations (cf. \cite[Section 2.4]{P01}),
and then bound the condition number $\kappa(V_{\bf s})$ by
applying Corollary \ref{coninv2}.

%------------------------------------------------------------------------------

\section{The Case of the Quasi-Cyclic Sequence of Knots}
\label{skntunit}

%------------------------------------------------------------------------------
%------------------------------------------------------------------------------

If  all the knots $s_0,\dots,s_{n-1}$
  
(a)  lie on the unit circle $C(0,1)$
and 

(b) are equally spaced on it, \\
then
we arrive at perfectly conditioned Van\-der\-monde matrices, such as
the DFT matrix $\Omega$, satisfying
$\kappa(\Omega)=1$. 
A rather minor deviations from assumption (b), however, can
make the matrix ill-conditioned, as \cite{G90} has shown empirically
and we deduce formally from Theorem \ref{thninv} next.

%------------------------------------------------------------------------------

\begin{example}\label{ex2} {\em The quasi-cyclic sequence of knots.}

For applications to interpolation and quadrature, one seeks
 a sequence of $n\times n$ well 
conditioned Van\-der\-monde matrices for $n=2,3,\dots$,
such that all the knots $s_i$ are reused recursively
 as $n$ increases (cf. \cite[Section IV]{G90}).
For $n=2^k$, we choose the $n\times n$ matrix of DFT at $n$ points,
defined by the equally spaced knots $s_i=\exp(2\pi i\sqrt {-1}/2^k)$ for
$i=0,1,\dots,2^k-1$.  We can keep these knots for 
all $n>2^k$ as well, but then how  should we choose
the remaining knots $s_i$ if
$n$ is not a power of 2?    
%The straightforward choice
%$s_0=1$, $s_i=\omega_{2^{k+1}}^{2(i+1-2^{k})-1}$, for $i=1,\dots,2^{k+1}-1$,
A rather straightforward choice, 
called the quasi-cyclic sequence in \cite[Section IV]{G90},
is the sequence  $s_i=\exp(2\pi f_i\sqrt {-1})$, $i=0,1,\dots,n-1$,
 defined by the following sequence   of fractions $f_i$,

    $$0, \tfrac{1}{2}, \tfrac{1}{4}, \tfrac{3}{4}, \tfrac{1}{8}, \tfrac{3}{8}, \tfrac{5}{8}, \tfrac{7}{8}, \tfrac{1}{16}, \tfrac{3}{16}, \tfrac{5}{16}, \tfrac{7}{16}, \tfrac{9}{16}, \tfrac{11}{16}, \tfrac{13}{16}, \tfrac{15}{16}, \ldots$$ 
In other words, one can use the DFT matrix for $n=2^k$
and positive integers $k$,
and then, as $n$ increases, can extend
 the set $s_0,\dots,s_{2^k-1}$ of the $2^k$-th roots of 1
recursively by adding, one by one, 
the remaining  $2^{k+1}$-st roots of 1 
 in counter-clockwise order.
For $n=2^k$, these are the unitary matrices of DFT,
but for $n=3\cdot 2^k$ and  $k\ge 4$, their condition numbers
become large, exceeding  
$1,000,000$  already for $n=48$ (cf.
\cite[Section IV]{G90}). 
This has been remaining an empirical observation since 1990, 
but next we  
deduce 
%from equation (\ref{eqvinvn}) as well as 
from Corollary \ref{coninv} 
that $\log(\kappa(V_{\bf s}))$ grows at least proportionally to $2^k$,
that is,  $\kappa(V_{\bf s})$ grows exponentially in $n$. 
\end{example}

%------------------------------------------------------------------------------

\begin{theorem}\label{thqsc} 
Let $V_q$ denote the $n\times n$ Van\-der\-monde matrix
defined by the quasi-cyclic sequence of knots on the unit circle  $C(0,1)$
for $n=3q$, $q=2^{k-1}$,   
and a positive integer $k$. Then $\kappa(V_q)\ge 2^{q/2}\sqrt n$.
\end{theorem}
\begin{proof}
For  $n=3q$,
the sequence has the $2q$ knots $s_i=\omega_{2q}^i$, $i=0,1,\dots,2q-1$,
 equally spaced on the circle $C(0,1)$, 
and the $q$ additional knots $s_i=\omega_{4q}^{2i+1}$, $i=0,1,\dots,q-1$,
equally spaced on the upper half of that circle 
such that $s(x)=(x^{2q}-1)p_q(x)$, for $p_q(x)=\prod_{i=0}^{q-1}(x-\omega_{4q}^{2i+1})$.  
Corollary \ref{coninv} implies that  
$\kappa(V_q)\ge |s(f)|\sqrt n/2$, where we  can choose any complex $f$
such that $|f|=1$. 

Choose  $f=-\sqrt {-1}~\omega_{4q}$ and then observe that
 $f^{2q}-1=-2$, and so in this case 
\begin{equation}\label{eqpqf}
|s(f)|=2|p_q(f)|.
\end{equation}
Furthermore,
$|f-\omega_{4q}^{2i+1}|\ge \sqrt 2$,
for $i=1,\dots,q$. Consequently, 
$|s(f)|\ge 2^{q/2+1}$, because
 \begin{equation}\label{eqspqf}
 \kappa(V_q)\ge 0.5 |s(f)| \sqrt n,
\end{equation}
and the theorem follows. 
\end{proof}

%------------------------------------------------------------------------------

The theorem already implies that the condition number
$\kappa(V_k)$ grows exponentially in $n$ for $n=3q$ and $q= 2^{k-1}$,
but we can strengthen the estimate a little further
because of the following observation.
\begin{fact}\label{faqcs}
$|f-\omega_{4q}^{2i+1}|\ge 2\cos ((0.5-i/q)\pi/2)$ for
$q$ being a power of two and 
$i$ in the range between $0$ and $q$.
\end{fact}

This fact implies that
the distance $|f-\omega_{4q}^{2i+1}|$ grows
from $\sqrt 2$ to 2 as 
$i$ grows  from 0 to $q/2$,
but this  distance  decreases back to $\sqrt 2$ as 
$i$ grows  from $q/2$
  to $q$. 
In the limit, as $q\rightarrow \infty$, we obtain 
\begin{equation}\label{eqqcslim}
|p_q(f)|\ge \exp(y)~{\rm for}~y=
\int_0^{q}\ln(2\cos((0.5-x/q)\pi/2))~{\rm d}x.
\end{equation}

For large integers $q$ this expression closely approximates lower bounds,
which we can obtain  for any $q$ by replacing the integral with a 
partial sum. Let us do this
for $q$ divisible by 12.
Replace the integral with a partial sum 
by applying Fact \ref{faqcs} 
for $i=q/6$ and $2\cos ((0.5-i/q)\pi/2)=\sqrt 3$
and obtain
\begin{equation}\label{eqqcs}
 |p_q(f)|\ge  18^{q/6}.
\end{equation}
Refine this estimate by applying  
 Fact \ref{faqcs} 
for $i=q/3$ and $2\cos ((0.5-i/q) \pi/2)\approx 1.9318517$
and deduce that
\begin{equation}\label{eqqcs1}
 |p_q(f)|\ge (2\cos (\pi/12) \sqrt 6)^{q/3}.
\end{equation}

For $n=48$ the lower bound of Theorem \ref{thqsc} on $\kappa(V_q)$ is just 
$2^8\sqrt {48}=1024\sqrt 3>1,773.6$, but  (\ref{eqqcs}) and 
(\ref{eqqcs1}) enable us to
increase this lower bound to 15,417  and 27,598, respectively. 
By engaging Fact  \ref{faqcs} at first for  $i=q/4$ and 
$2\cos ((0.5-i/q) \pi/2)\approx 1.9615706$
and then 
for  $i=q/12$ and $2\cos ((0.5-i/q) \pi/2)\approx 1.586707$,
we increase these lower bounds at first to 38,453
and then to 71,174.
We leave further refinement of the bounds as an exercise for the reader. 

Can one still define a sequence of well-conditioned $h\times h$ Van\-der\-monde matrices 
for $h=1,\dots,n$, which would include all the $2^k$th roots of 1 
for $k=\lfloor\log_2 n\rfloor$
and would consists of only $n$ knots overall?
 \cite[Section 4]{G90} proves that such  a sequence of well-conditioned $h\times h$ Van\-der\-monde matrices is obtained if we  define
the knots $s_i=2\pi f_i\sqrt {-1}$, $i=0,1,\dots,n-1$,
 by  the following {\em van der Corput sequence} of fractions $f_i$,
which  reorders the knots of quasi-cyclic sequence in
a zigzag manner,
    $$0, \tfrac{1}{2}, \tfrac{1}{4}, \tfrac{3}{4}, \tfrac{1}{8}, \tfrac{5}{8}, \tfrac{3}{8}, \tfrac{7}{8}, \tfrac{1}{16}, \tfrac{9}{16}, \tfrac{5}{16}, \tfrac{13}{16}, \tfrac{3}{16}, \tfrac{11}{16}, \tfrac{7}{16}, \tfrac{15}{16}, \ldots.$$ 
 
The knots $s_1,\dots,s_n$ are equally spaced on the  unit circle $C(0,1)$
  for $n=2^k$ being the powers of 2 and are distributed on it 
quite evenly for all $n$. Gautschi  has proved the following estimate
in  \cite[Section 4]{G90}.

\begin{theorem}\label{thgaut}
Given the van der Corput sequence  of $n\times n$
Van\-der\-monde matrices $V_{\bf s}$ for $n=1,2,3,\dots$, it holds that
 $\kappa(V_{\bf s})<\sqrt {2n}$ for all $n$.  
\end{theorem}

%------------------------------------------------------------------------------

\section{Numerical Instability of Gaussian Elimination  with No Pivoting for the Matrices
of Discrete Fourier Transform}\label{sgenpv}

%------------------------------------------------------------------------------
 
Gaussian elimination with partial pivoting,
that is, with appropriate row interchange,
is extensively used in matrix computations,
even though pivoting is a substantial burden  
in the context of the modern computer technology.
The papers \cite{PQY15} and \cite{PZ15} 
list the following problems with pivoting:
it 
interrupts the stream of arithmetic operations 
with foreign operations of comparison,
involves book-keeping, compromises data locality, 
complicates parallelization of the computations,  and
increases communication overhead and data dependence.
The user prefers to apply
Gaussian elimination with no pivoting
(hereafter 
we use the acronym
{\em  GENP}) unless 
it fails or runs into numerical problems, 
which occurs
if and only if some leading blocks 
of the input matrix $A$ are singular or  
ill-conditioned (cf. \cite[Theorem 5.1]{PQZ13}). 

These observations have motivated substantial 
interest to numerical application of GENP 
and the efforts towards
supporting  it with 
various heuristic and randomization techniques
(see \cite{PQY15} and \cite{BCD14}
for history accounts).
Our next theorem contributes 
 to the study of this
important subject.
		
% - - - - - - - - - - - - - - - - - - - - - - - - - - - - - - - - - - - - -

\begin{theorem}\label{thgenpdft}
Assume for convenience that  $n$ is even
and let $\Omega^{(q)}$ denote
the $q\times q$ leading (that is, northwestern) block of 
the matrix $\Omega$ for $q= n/2$.
Then
$\kappa(\Omega^{(q)})\ge 2^{n/4-1}\sqrt n$.
\end{theorem}

\begin{proof}
As follows from Corollary \ref{coninv},
 $\kappa(\Omega^{(q)})\ge |s(f\omega_q^i)|\sqrt n/2$
where $s(x)=\prod_{j=0}^{q-1}(x-\omega_n^j)$ and where we can choose 
$f=-\sqrt {-1}$.
Under such a choice $|f-\omega_n^j|\ge \sqrt 2$ for $j=0,\dots,q-1$,
and so $|s(f)|\ge 2^{n/4}$. Hence 
$\kappa(\Omega_n^{(q)})\ge 2^{n/4}\sqrt n/2$,
 and Theorem \ref{thgenpdft}  follows.
\end{proof}

The estimates of equations (\ref{eqspqf})--(\ref{eqqcs1})
still hold if we substitute $p_q(x)=s(x)$ and $n=2q$, which  should replace equations (\ref{eqpqf}) and
$n=3q$ of the previous section, 
and so we can 
 strengthen the bound of  Theorem \ref{thgenpdft}
accordingly.

The techniques of the  proof of Theorem \ref{thgenpdft}
can be extended to reveal that all the $k\times k$ blocks 
of the unitary matrix $\frac{1}{\sqrt {n}}\Omega$ 
are badly ill-conditioned as long as the integers $k$ or $n-k$
are large. This observation can 
lead to interesting conclusion about randomized 
structured preprocessing of GENP,
studied in \cite{PQZ13}, \cite{PQY15}, \cite{PZ15},  and \cite{PZa}.
Namely, together with 
 \cite[Theorem 6.2]{PQY15} and  \cite[Theorem 9]{PZ15}
Theorem \ref{thgenpdft} implies that 
Gaussian random circulant multipliers  
support numerically stable application of GENP to the matrix $\Omega$
only with a probability near 0,
even though they have consistently supported numerical GENP 
in the extensive tests with  matrices of various other classes 
in \cite{PQZ13} and
\cite{PQY15}.

Finally, in some contrast to Theorem \ref{thgenpdft},
one can turn all
leading blocks of   the matrix $\Omega$
into well-conditioned  matrices,
by re-ordering its rows according to 
the van der Corput sequence.

\begin{theorem}\label{thdftldvd}
 Given the van der Corput sequence  of $n\times n$
Van\-der\-monde matrices $V_{\bf s}$ for $n=1,2,3,\dots$, 
the leading block of any of the matrices in the sequence
is also a matrix in the same sequence.
\end{theorem}

Combine this observation with Theorem \ref{thgaut}
and conclude that all leading blocks of a Van\-der\-monde matrix 
defined by a
van der Corput sequence of knots are nonsingular and well-conditioned,
and so
application of GENP to such a Van\-der\-monde matrix is numerically stable.

%------------------------------------------------------------------------------

\section{Low-Rank Approximation of Cauchy and CV Matrices}\label{scv1}

%------------------------------------------------------------------------------
%------------------------------------------------------------------------------ 

Next,
under some mild assumptions on the knots $s_0,\dots,s_{m-1}$,
 we deduce an exponential low bound on the norm  of the inverse of 
a nonsingular  CV matrix $C_{{\bf s},f}$, which is the
reciprocal of its 
smallest singular value,
$||C_{{\bf s},f}^{-1}||=1/\sigma_{n}(C_{{\bf s},f})$.
At first we deduce such a bound for a submatrix and
then  extend it to the  matrix itself by applying Corollary   \ref{coappnd}.
In the next section, by extending this
estimate, we prove our alternative exponential 
lower bound on the norm $||V_{\bf s}^{-1}||$ 
and the condition number $\kappa(V_{\bf s})$ of  matrix $V_{\bf s}$,
which holds if the knot sets of the matrix 
$V_{\bf s}$
deviates substantially from the
evenly spaced distribution on the circle $C(0,1)$.

\begin{definition}\label{defss} (See  \cite[page 1254]{CGS07}.)

Two complex points $s$ and $t$
are $(\eta,c)$-{\em separated}, 
for $\eta>1$ and a complex {\em center} $c$,
if $|\frac{t-c}{s-c}|\le 1/\eta$.

Two   complex sets  $\mathcal S$ and 
$\mathcal T$ are $(\eta,c)$-{\em separated} 
if  every pair of points  $s\in \mathcal S$
and $t\in \mathcal T$ 
is $(\eta,c)$-separated.
%$\delta_{c,\mathcal S}=\min_{s\in ~\mathcal S}|s-c|$ 
%$\delta_{c,\mathcal T}=\min_{t\in \mathcal T}|t-c|$
%denotes the distance from the center $c$ to the set $\mathcal S$.
%and $\mathcal T$, respectively.
\end{definition}

%Hereafter $|\mathcal B|$ denotes the  cardinality of a set $\mathcal B$.

\begin{theorem}\label{thss0}  (Cf. \cite[Section 2.2]{CGS07},
\cite[Theorem 29]{P15}.)
Let  $m$ and $l$ denote 
 a pair of positive integers and 
let 
$\mathcal S_m=\{s_0,\dots,s_{m-1}\}$ and 
$\mathcal T_l=\{t_0,\dots,t_{l-1}\}$
denote  two sets 
that are $(\eta,c)$-separated 
for a pair of 
 real $\eta>1$ and  complex $c$.
Define 
the Cauchy matrix
$C=(\frac{1}{s_i-t_j})_{i,j=0}^{m-1,l-1}$. 
Write
 $\delta=\delta_{c,\mathcal S}=\min_{i=0}^{m-1} |s_i-c|$.
%denote the distance from the center $c$ to the set $\mathcal S$.
%$\delta_{c,\mathcal T}=\min_{j=0}^{n-1} |t_j-c|$
%(cf. Definition \ref{defss}).
Then
 
%\begin{equation}\label{eqappr}
$$1/\sigma_{\rho}(C)\ge (1-\eta)\eta^{\rho-1}\delta~
{\rm for~all~positive~integers}~\rho.$$
%\end{equation}
\end{theorem}

We  
 apply the theorem to 
any knot set
$\mathcal S_m$ and
to some specific knot sets $\mathcal T_l$.

\begin{theorem}\label{thss} 
Let us be given 
a complex  $f$  such that $|f|=1$ and
two knot sets $\mathcal S_m=\{s_0,\dots,s_{m-1}\}$
and $\mathcal T_n=\{t_0,\dots,t_{n-1}\} $
where
$t_j=f\omega^j$, for $j=0,\dots,n-1$,  
$\omega=\exp(2\pi \sqrt{-1}/n)$.
Define
 the Cauchy matrix
$C_{m,n}=(\frac{1}{s_i-t_j})_{i\in \mathcal S_{m}, j\in \mathcal T_n}$.

Fix  two integers
$j'$ and $j''$ in the range $[0,n-1]$ such that
$0<l= j''-j'+1\le n/2$ and 
define 
the subset 
 $\mathcal T_{j',j''}=\{f\omega^j\}_{j=j'}^{j''}$ of the set $\mathcal T_n$.

%fix $n$ complex knots $s_0,\dots,s_{n-1}$, and define the CV matrix 
%$C_{{\bf s},f}=(\frac{1}{s_i-t_j})_{i,j=0}^{n-1}$.

Let $c=0.5(t_{j'}+t_{j''})$
denote the midpoint of the line interval 
with the endpoints $t_{j'}$ and $t_{j''}$
and let  $r=|c-t_{j'}|=|c-t_{j''}|=\sin((j''-j')\frac{\pi}{n})$ 
denote the distance from this midpoint to the points 
$t_{j'}$ and $t_{j''}$.
(Note that 
%\end{equation}
$(j''-j')\frac{\pi}{n}$ is a close upper bound on $r$ for large integers $n$.) 

Fix a constant $\eta>1$ and partition the set $\mathcal S_m$
into the two subsets $\mathcal S_{m,-}$ and $\mathcal S_{m,+}$
of cardinalities $m_-$ and $m_+$, respectively,
such that $m=m_-+m_+$, the knots of $\mathcal S_{m,-}$ lie in  
the open disc
$D(c,\eta r)=\{z:~|z-c|< \eta r\}$,
and the knots of $\mathcal S_{m,+}$ lie in the exterior of that disc.
 
(i) Then the two sets $\mathcal S_{m,+}$ and $\mathcal T_{j',j''}$
are $(c,\eta)$-separated and

\medskip

(ii) $1/\sigma_{\rho+m_-+n-l}(C_{m,n}) \ge \eta^{\rho}(\eta-1)r$
for all positive integers $\rho$.
\end{theorem}

\begin{proof} 
Part (i) follows because 
${\rm distance}(c,\mathcal T_{j',j''})=r$
and because ${\rm distance}(c,\mathcal S_{m,+})\ge \eta r$
by virtue of the definition of the set $\mathcal S_{m,+}$. 

Now define the two following submatrices of the matrix $C_{m,n}$,
 $$C_{m,j',j''}=\Big (\frac{1}{s_i-t_j}\Big )_{i\in \mathcal S_{m}, j\in \mathcal T_{j',j''}}~
{\rm and}~C_{m_+,j',j''}=\Big (\frac{1}{s_i-t_j}\Big )_{i\in \mathcal S_{m,+}, j\in \mathcal T_{j',j''}}.$$

Apply Theorem \ref{thss0} to the matrix $C=C_{m_+,j',j''}$  and
$\delta=\eta r$ and obtain that
$$1/\sigma_{\rho}(C_{m_+,j',j''}) \ge \eta^{\rho}(\eta-1)r~{\rm  
~for~all~positive~integers}~\rho.$$

Combine this bound with Corollary   \ref{coappnd} applied
 for $k=m_-$ 
and obtain that
$$1/\sigma_{\rho+m_-}(C_{m,j',j''})\ge \eta^{\rho}(\eta-1)r
~{\rm  
~for~all~positive~integers}~\rho.$$

Combine  the latter bound with Corollary   \ref{coappnd} applied
 for $k=n-l$ 
and obtain part (ii) of the theorem.
\end{proof} 

%Substitute  $m=n$ and $\bar \rho=l-m_-$ into part (v)
%and obtain the following result.

\begin{corollary}\label{coss1}
Under the assumptions of Theorem \ref{thss}, let $m=n$ and $\bar \rho=l-m_->0$.
Then $C_{{\bf s},f}=C_{m,n}$ and

$$||C_{{\bf s},f}^{-1}||=1/ \sigma_{n}(C_{{\bf s},f})\ge \eta^{\bar\rho}(\eta-1)r.$$
\end{corollary}

%------------------------------------------------------------------------------

\section{A Typical Large Van\-der\-monde Matrix is Badly Ill-conditioned}\label{svndbd}

%------------------------------------------------------------------------------

Combine part (i) of Theorem \ref{thninv} and
 Corollary
   \ref{coss1} and obtain 
 the following result.

\begin{corollary}\label{covnd1}
Under the assumptions of  Corollary \ref{coss1}, let $s_+\le 1$
for $s_+$ of equation (\ref{eqcond}). Then

$$\kappa(V_{\bf s})\ge  \eta^{\bar \rho}(\eta-1)r\sqrt n/2.$$
\end{corollary}

%-------------------------------------------------------------------------------

The corollary shows that, under its mild assumptions,
 the condition number 
$\kappa(V_{\bf s})$ of 
an  $n\times n$ Van\-der\-monde matrix 
$V_{\bf s}$ grows exponentially in $\bar \rho$, 
and thus also in $n$  
if $\bar \rho$ grows proportionally to $n$.

To specify an implication, 
fix any real $\eta>1$ and any pair
of integers $j'$ and $j''$ in the range $[0,n-1]$
such that $0<j''-j'\le n/2$. Then 
the disc $D(c,r)$ contains precisely 
$l=j''-j'+1$ knots from the set $\mathcal T$,
and $m_-$ denotes the number of knots of the set 
$\mathcal S$ in the disc  $D(c,\eta r)$.
Hence Corollary \ref{covnd1} implies that
$\kappa(V_{\bf s})$ is exponential in $n$
unless $l-m_-=o(n)$, for all such triples 
$(\eta,j',j'')$ with $l=j''-j'+1$ of order $n$,
 that is, unless the number of knots $s_i$ in the 
disc $D(c,\eta r)$ exceeds,
 matches or nearly matches the number of knots $t_j$ 
in the disc $D(c,r)$ for all such triples. 
  Clearly, this property is  only satisfied
 for all such triples $(\eta,j',j'')$ if
 the associated knot set $\mathcal S$ is more or less evenly 
spaced  on or about the unit circle $C(0,1)$.

%------------------------------------------------------------------------------

\section{Numerical Tests}\label{ststs}

%------------------------------------------------------------------------------

Numerical tests have been performed by Liang Zhao 
in the Graduate Center of the City University of New York.
He has run them on a Dell server by using Windows system and MATLAB R2014a
(with the IEEE standard double precision).

%\subsection{Condition Number of Van\-der\-monde Matrices with Some Knots Isolated
%from the Unit Circle}

Table \ref{VandCond} displays 
the condition numbers $\kappa(V_{\bf s})$ of $n\times n$ Van\-der\-monde matrices,
 with the knots $s_i = \omega_n^i$ being the $n$th 
roots of 1, for $i = 0, 1,\dots, n-2$,  and
with the knot $s_{n-1}$  in the range from 1.14 to 10.
For comparison the table also  displays
 the lower estimates of part (i)
of Corollary \ref{cocndvceasy}
 in the case of $m=n$. 
By setting   $s_{n-1}=\omega_n^{n-1}$ we would have 
$\kappa(V_{\bf s})=1$, but with  the growth of the value $|s_{n-1}|$
and even with the growth of $n$ for $s_{n-1}$ fixed at 1.14,
 the matrix was quickly becoming ill-conditioned.

Table \ref{VandCond1} displays the 
condition numbers  of $n\times n$ Van\-der\-monde matrices defined 
by  various pairs of integers $k$ and $n$ and the knots $s_0,\dots, s_{n-1}$
such that
 $s_i = \omega_{n-k}^i$ are the $(n-k)$th 
roots of 1, for  $i = 0, 1,\dots, n-k-1$, and 
 $s_{n-k+i} = \rho\omega_{k}^i$ are the $k$th 
roots of 1  scaled by $\rho$, for  $i = 0, 1,\dots, k$ and $\rho=3/4,1/2$.
For comparison the table also  displays
 the lower estimates 
$\kappa_-= |V_{\bf s}|\nu^{k-1}/(\sqrt k\max \{k,\nu/(\nu-1)\})$ of part (ii)
of Corollary \ref{cocndvceasy}, for $\nu=1/\rho$.

\begin{table}[h]
\caption{Condition Numbers of Van\-der\-monde Matrices with a Single Absolutely Large Knot}
\label{VandCond}
\begin{center}
\begin{tabular}{|c|c|c|c|}
\hline
n 	&	 $s_{n-1}$ 	& $\kappa(V_{\bf s})$ & $s_+^{n-1}/\sqrt n$	\\ \hline
64	&	1.14E+00	&	3.36E+03	&	4.98E+02	\\ \hline
64	&	1.56E+00	&	6.88E+11	&	2.03E+11	\\ \hline
64	&	3.25E+00	&	4.62E+32	&	2.22E+31	\\ \hline
64	&	1.00E+01	&	5.47E+63	&	1.25E+62	\\ \hline
128	&	1.14E+00	&	1.08E+07	&	1.60E+06	\\ \hline
128	&	1.56E+00	&	3.24E+32	&	3.64E+23	\\ \hline
128	&	3.25E+00	&	6.40E+76	&	9.03E+63	\\ \hline
128	&	1.00E+01	&	1.30E+130	&	8.84E+125	\\ \hline
256	&	1.14E+00	&	1.57E+14	&	2.33E+13	\\ \hline
256	&	1.56E+00	&	6.61E+62	&	1.66E+48	\\ \hline
256	&	3.25E+00	&	2.89E+132	&	2.12E+129	\\ \hline
256	&	1.00E+01	&	1.18E+265	&	6.25E+253	\\ \hline

\end{tabular}
\end{center}
\end{table}

\begin{table}[h]
\caption{Condition Numbers of Van\-der\-monde Matrices 
with $k$ Absolutely Small Knots}
\label{VandCond1}
\begin{center}
\begin{tabular}{|c|c|c|c|}
\hline
& &    $\rho=3/4$   &  $\rho=1/2$ \\ \hline
n 	&	 $k$ 	 & $\kappa(V_{\bf s}),\kappa_-$
& $\kappa(V_{\bf s}),\kappa_-$	\\ \hline
64	&	8	 	&	4.04E+01, 7.14E+00 	&	6.90E+02, 2.44E+02 	\\ \hline
64	&	16		&	2.71E+02, 4.78E+01 	&	1.19E+05, 4.19E+04 	\\ \hline
64	&	32	 	&	1.71E+04, 3.02E+03 	&	4.91E+09, 1.74E+09 	\\ \hline
128	&	8	 	&	5.85E+01, 1.03E+01 	&	1.00E+03, 3.53E+02 	\\ \hline
128	&	16	 	&	4.03E+02, 7.13E+01 	&	1.77E+05, 6.24E+04 	\\ \hline
128	&	32		&	2.70E+04, 4.77E+03 	&	7.77E+09, 2.75E+09 	\\ \hline
256	&	8		&	8.38E+01, 1.48E+01 	&	1.43E+03, 5.06E+02 	\\ \hline
256	&	16	 	&	5.85E+02, 1.03E+02 	&	2.56E+05, 9.05E+04 	\\ \hline
256	&	32		&	4.02E+04, 7.11E+03 	&	1.16E+10, 4.09E+09 	\\ \hline

\end{tabular}
\end{center}
\end{table}

%\subsection{Condition Numbers of the Leading Blocks of DFT Matrices}

Table \ref{tabquasi} displays the condition numbers 
$\kappa(V_{\bf s})$ of the $3q\times 3q$ 
Van\-der\-monde Matrices defined by quasi-cyclic sequence 
of knots for $q=2^{k}$ and integers $k$.
%LZ
The table also displays
 the lower bounds, $\kappa$ and $\kappa'$,  on these numbers 
calculated based on equations (\ref{eqqcs1}) and (\ref{eqqcslim}),
respectively.

\begin{table}[h]
\caption{Condition Numbers of Van\-der\-monde Matrices 
defined by quasi-cyclic sequence 
of knots}
\label{tabquasi}
\begin{center}
\begin{tabular}{|c|c|c|c|c|}
\hline
n	&	q 	& $\kappa(V_{\bf s})$ & $\kappa$	&	$\kappa'$	\\ \hline				
%6  & 2 & 3.70e+00 & 4.90e+00 & 9.06E+00 \\ \hline
12  & 4 & 2.16e+01 & 1.96e+01 & 1.03E+01\\ \hline 
24  & 8 & 1.50e+03 & 2.22e+02 & 1.06E+02\\ \hline 
48  & 16 & 1.16e+07 & 2.01e+04 & 1.13E+04\\ \hline 
96  & 32 & 9.86e+14 & 1.16e+08 & 1.27E+08\\ \hline

%128	&	8.51E+35	&	6.34E+28	\\ \hline
%256	&	7.91E+71	&	4.02E+57	\\ \hline

\end{tabular}
\end{center}
\end{table}

Table \ref{LeadDFTBlock} displays the condition numbers 
$\kappa(\Omega^{(q)})$ of the $q\times q$ leading blocks 
of the $n\times n$ DFT matrices  $\Omega=(\omega_n^{ij})_{i,j=0}^{n-1}$, for
%LZ
$n = 2q=8, 16, 32, 64$.  
For the same pairs of $n$ and $q$, the table also displays
 the lower bounds, $\kappa_-$ and $\kappa'_-$,  on these numbers
calculated based on Theorem \ref{thgenpdft} and  equation (\ref{eqqcslim})
for $s(x)=p_q(x)$, 
respectively.

 The test results show  that the  lower bounds based on
equations (\ref{eqqcs1}) and  (\ref{eqqcslim}) and even the ones based on 
Theorem \ref{thgenpdft}
grew very fast, but the
condition numbers grew even faster.

\begin{table}[h]
\caption{Condition Numbers of the Leading Blocks of DFT matrices}
\label{LeadDFTBlock}
\begin{center}
\begin{tabular}{|*{6}{c|}}
\hline
n	&	 q	&  $\kappa(\Omega_n^{(k)})$ 	&	$\kappa_-$ 	& $\kappa'_-$	\\ \hline
8	&	4	& 	1.53E+01	& 	8.00E+00	& 	1.03E+01	\\ \hline
16	&	8	&	1.06E+03	&	4.53E+01	&	1.06E+02	\\ \hline
32	&	16	&	8.18E+06	&	1.02E+03	&	1.13E+04	\\ \hline
64	&	32	&	8.44E+14	&	3.71E+05	&	 1.27E+08	\\ \hline
%128	&	64	&		&	3.44E+10	&	4.26E+30	\\ \hline
%256	&	128	&	8.84E+16	&	2.09E+20	&	1.81E+61	\\ \hline
%512	&	256	&	2.19E+19	&	5.44E+39	&	3.28E+122	\\ \hline
%%1024	&	512	&		&	2.62E+78	&	1.08E+245	\\ \hline

\end{tabular}
\end{center}
\end{table}

%\subsection{GENP with DFT Matrices}
 Table \ref{GENP} shows the mean and standard deviation of the relative 
residual norms $rn = ||\Omega_k {\bf x} - {\bf b}||/||{\bf b}||$
when 
GENP was applied to the linear systems of $n$ equations   
$\Omega_n{\bf x}={\bf b}$, for the $n\times n$ DFT matrix $\Omega_n$, 
vectors ${\bf b}$ of dimension $n$ filled with i.i.d. standard Gaussian
random variables,
$n = 64, 128, 256, 512, 1024$,
and
 100 linear systems $\Omega_n {\bf x} = {\bf b}$ have been 
solved for each $n$.  
The tests have confirmed that the relative residual norms
 grew very fast for these inputs as $n$ increased.

\begin{table}[h]
\caption{GENP with DFT Matrices: Relative 
Residual Norms}
\label{GENP}
\begin{center}
\begin{tabular}{|c|c|c|}
\hline
{\bf n}	&	\bf mean	&	\bf std	       \\ \hline
16	&	8.88E-14	&	2.75E-14	\\ \hline
32	&	8.01E-10	&	3.56E-10	\\ \hline
64	&	5.31E-03	&	1.83E-03	\\ \hline
128	&	5.00E+00	&	2.15E+00	\\ \hline
256	&	7.87E+02	&	5.20E+02	\\ \hline
512	&	7.05E+03	&	4.44E+03	\\ \hline
1024	&	6.90E+04	&	3.26E+04	\\ \hline

\end{tabular}
\end{center}
\end{table}

%\clearpage 

%------------------------------------------------------------------------------

\section{Conclusions}\label{sconc}

%------------------------------------------------------------------------------

Our estimates for the condition numbers 
 of $n\times n$ Van\-der\-monde matrices $V_{\bf s}=(s_i^j)_{i,j=0}^{n-1}$,
having complex knots $s_0,\dots,s_{n-1}$,
provide some missing formal support for the well known 
empirical observation
that
such matrices tend to be badly ill-conditioned
already for moderately large integers $n$.
Our results also indicate that all the
 exceptions are  narrowed to the  Van\-der\-monde matrices
(such as  
the matrices of the discrete Fourier transform)
whose knots are more or less equally spaced
on or about the unit circle $\{x:~|x=1|\}$.  
Our auxiliary  inversion formulae 
for  Cauchy and Van\-der\-monde matrices
%of (\ref{eqvinv}) and (\ref{eqinvcf})
may be of 
independent interest.

We began our study with readily deducing from the Ekkart--Young Theorem that  
the condition number $\kappa(V_{\bf s})$
is
exponential in $n$
unless all knots $s_0,\dots,s_{n-1}$
lie in or near the unit disc $\{x:~|x|\le 1\}$
and furthermore unless they lie
mostly on or near the  unit circle $C(0,1)$.

Then, by applying more involved techniques, based on 
the  Van\-der\-monde--Cauchy transformation of matrix structure,
our formulae for the  inversion of a Cauchy matrix,
their extension to Van\-der\-monde matrix inversion,
and the techniques of low-rank approximation of Cauchy matrices,
 we prove our exponential lower bounds on  
the condition numbers 
of Van\-der\-monde matrices
$V_{\bf s}$,
which hold unless   
 the  knots $s_0,\dots,s_{n-1}$
are more or less equally spaced on or near the 
unit circle $C(0,1)$.

Our test results are in a rather good accordance 
with our estimates. 
One can try to narrow the remaining gap between
the  estimated
and observed values of the condition number 
$\kappa(V_{\bf s})$
by varying the  parameters $t_j$ and $f$ that we substitute into our basic
equation (\ref{eqinvc}).

All our estimates hold for the transpose of a Van\-der\-monde matrix as well.
Furthermore we can extend our upper bounds on the smallest positive 
singular value $\sigma_{n}(V_{\bf s})=1/||V_{\bf s}||$ to the upper bounds 
$$\sigma_{n}(V)\le \sigma_{n}(V_{\bf s})~\sum_{j=1}^d||C_j||~||D_j||~{\rm and}~
\sigma_{n}(C)\le \sigma_{n}(C_{\bf s,t})~\sum_{j=1}^d||D'_j||~||D''_j||$$
on the smallest positive singular values $\sigma_{n}(V)$ and 
$\sigma_{n}(C)$, respectively, of the 
 ``Van\-der\-monde-like" and ``Cauchy-like" matrices of the important classes
$V=\sum_{j=1}^d~D_jV_{\bf s}C_j$ and 
$C=\sum_{j=1}^d~D'_jC_{\bf s,t}D''_j$, respectively,
for a bounded positive integer $d$,
diagonal matrices $D_j$, $D_j'$, and $D''_j$, 
and factor-circulant (structured) matrices $C_j$,
$j=1,\dots,d$ (cf. \cite[Examples 1.4.1 and 4.4.6]{P01}).

Most of our results can be readily extended to 
rectangular  Van\-der\-monde, Cauchy, ``Van\-der\-monde-like",
 and  ``Cauchy-like" matrices, and
a natural research challenge is their extension to
 polynomial  Van\-der\-monde matrices
$V_{\bf P,s}=(p_j(x_i))_{i,j=0}^{m-1,n-1}$
where ${\bf P}=(p_j(x))_{j=0}^{n-1}$
is any basis in the space of polynomials of degree less than $n$,
for example, the basis made up of Chebyshev polynomials.
Equation  (\ref{eqinvc}) is extended to this 
case as follows (cf. \cite[equation (3.6.8)]{P01}),

$$C_{\bf s,t}=\Big (\frac{1}{s_i-t_j}\Big )_{i,j=0}^{m-1,n-1}=\diag(t(s_i)^{-1})_{i=0}^{m-1}V_{\bf P,s}V^{-1}_{\bf P,t}\diag(t'(t_j))_{j=0}^{n-1}.$$
Here $t(x)=\prod_{j=0}^{n-1}(x-t_j)$
and ${\bf t}=(t_j)_{j=0}^{n-1}$ denotes the coefficient vector 
of the polynomial $t(x)-x^n$.
The decisive step could be the selection of a proper set of knots $\mathcal T=\{t_j\}_{j=0}^{n-1}$,
 which would enable us to reduce the estimation of 
the condition number of 
a polynomial Van\-der\-monde matrix $V_{\bf P,s}$
to the same task for a
Cauchy matrix $C_{\bf s,t}$. \cite[Sections 5 and 6]{G90}
might give us some initial guidance for the selection of the knot set, 
that would extend
the set of equally spaced knots on the unit circle, used 
in the case of Van\-der\-monde matrices. Then it would remain to extend 
our estimates for CV matrices to the new subclass of Cauchy matrices. 

%------------------------------------------------------------------------------

{\bf Acknowledgements:} 
  This work has been supported by NSF Grant CCF--1116736 
and PSC CUNY Award 67699-00 45. 
%I am also grateful 
%to the reviewers, for thoughtful and  helpful comments.

%------------------------------------------------------------------------------

\end{document}